\documentclass[a4paper,10pt]{article}
\usepackage{latexsym,amsmath,amsthm,amssymb}
\usepackage{a4wide}
\usepackage{hyperref}
\usepackage{marginnote}
\usepackage{color}
\usepackage{cite}
\hypersetup{
pdftitle={HardyMoserTrudinger}   
pdfauthor={Van Hoang Nguyen},
colorlinks = true,
linkcolor = magenta,
citecolor = blue,
}

\theoremstyle{plain}
\newtheorem{theorem}{Theorem}[section]

\newtheorem{lemma}[theorem]{Lemma}

\theoremstyle{definition}

\theoremstyle{remark}

\renewcommand{\thefootnote}{\arabic{footnote}}

\def\R{\mathbb R}

\def\al{\alpha}
\def\om{\omega}
\def\Om{\Omega}
\def\be{\beta}
\def\ga{\gamma}
\def\de{\delta}
\def\De{\Delta} 

\def\lam{\lambda}
\def\vphi{\varphi}
\def\na{\nabla}
\def\pa{\partial}
\def\lt{\left}
\def\rt{\right}

\def\i0i{\int_0^\infty}

\def\B{\mathbb B}
\def\dvH{dv_{\mathbb H^n}}

\numberwithin{equation}{section}

\title{The sharp Hardy--Moser--Trudinger inequality in dimension $n$}
\author{Van Hoang Nguyen
}

\begin{document}
\maketitle



\renewcommand{\thefootnote}{}

\footnote{Email: \href{mailto: Van Hoang Nguyen <vanhoang0610@yahoo.com>}{vanhoang0610@yahoo.com}} 

\footnote{2010 \emph{Mathematics Subject Classification\text}: 26D10, 35A23, 46E35}

\footnote{\emph{Key words and phrases\text}: Moser--Trudinger inequality, Hardy inequality, Hardy--Moser--Trudinger inequality, sharp constant, Green functions}

\renewcommand{\thefootnote}{\arabic{footnote}}
\setcounter{footnote}{0}

\begin{abstract}
In this paper, we prove a Hardy--Moser--Trudinger inequality in the unit ball $\B^n$ in $\R^n$ which improves both the classical singular Moser--Trudinger inequality and the classical Hardy inequality at the same time. More precisely, we show that for any $\beta \in [0,n)$ there exists a constant $C>0$ depending only on $n$ and $\beta$ such that
\[
\sup_{u\in W^{1,n}_0(\B^n), \mathcal H(u) \leq 1}\int_{\B^n} e^{(1-\frac\be n)\alpha_n |u|^{\frac n{n-1}}} |x|^{-\beta} dx \leq C 
\]
where $\alpha_n = n \om_{n-1}^{\frac1{n-1}}$ with $\om_{n-1}$ being the surface area of the unit sphere $S^{n-1} = \pa B$, and
\[
\mathcal H(u) = \int_{\B^n} |\na u|^n dx -\lt(\frac{2(n-1)}n\rt)^n \int_{\B^n} \frac{|u|^n}{(1-|x|^2)^n} dx.
\]
This extends an inequality of Wang and Ye in dimension two to higher dimensions and to the singular case as well. The proof is based on the method of transplantation of Green's functions and without using the blow-up analysis method. As a consequence, we obtain a singular Moser--Trudinger inequality in the hyperbolic spaces which confirms affirmatively a conjecture by Mancini, Sandeep and Tintarev \cite[Conjecture $5.2$]{MST}. We also propose an inequality which extends the singular Hardy--Moser--Trudinger inequality to any bounded convex domain in $\R^n$ which is analogue of the conjecture of Wang and Ye in higher dimensions.
\end{abstract}

\section{Introduction}
It is well-known that the Sobolev embedding is a basic and important tool in many aspects of Mathematics such as Analysis, Geometry, Partial of Differential Equations, Calculus of Variations, etc. Let $\Om$ be a bounded domain in $\R^n$ with $n \geq 2$ and $p\in (1,\infty)$. We denote by $W^{1,p}_0(\Om)$ the usual first order Sobolev space on $\Om$ which is the completion of the space $C_0^\infty(\Om)$ under the Dirichlet norm $\|\na u\|_{L^p(\Om)} := \lt(\int_\Om |\na u|^p dx\rt)^{\frac1p}$, $u\in C_0^\infty(\Om)$. For $1< p< n$, we have the following well-known Sobolev inequality
\begin{equation}\label{eq:Sobolev}
C \lt(\int_\Om |u|^q dx\rt)^{\frac1q} \leq \|\na u\|_{L^p(\Om)},\quad u\in W^{1,p}_0(\Om)
\end{equation}
for any $1 \leq q < p^* = np/(n-p)$ where $C>0$ is a constant depending only on $n,p,q$ and $\Om$. In other words, we have the embedding $W^{1,p}_0(\Om) \hookrightarrow L^q(\Om)$ for $q\in [1, p^*]$. However, in the limit case $p =n$ (thus, $p^* =\infty$) the embedding $W^{1,n}_0(\Om) \hookrightarrow L^\infty(\Om)$ fails. In this situation, it was proved independently by Yudovi${\rm \check{c}}$ \cite{Y}, Poho${\rm \check{z}}$aev \cite{P}, and  Trudinger \cite{T} that $W_0^{1,n}(\Om)$ can be embedded into an Orlicz space $L_{\varphi_n}(\Om)$ generated by the Young function $\varphi_n(t) = e^{c |t|^{\frac n{n-1}}} -1$ for some $c >0$. Later, Moser \cite{M} sharpened this result by finding out the sharp exponent $c$. More precisely, we have the following Moser--Trudinger inequality
\begin{equation}\label{eq:classicalMT}
\sup_{u \in W_0^{1,n}(\Om), \|\na u\|_{L^n(\Om)} \leq 1} \int_\Om e^{\al |u|^{\frac n{n-1}}} dx < \infty,
\end{equation}
if and only if $\al \leq \al_n: = n \om_{n-1}^{\frac1{n-1}}$ where $\om_{n-1}$ denotes the surface area of the unit sphere in $\R^n$. 

The Moser--Trudinger inequality plays the role of the Sobolev inequality in the limit case with many applications in many branches of Mathematics such as Analysis, Geometry and Partial Differential Equations, especially in studying the quasi-linear equations with exponential growth nonlinearity. It has been become an interesting subject to study. In fact, there have been many generalizations of the Moser--Trudinger inequality in many directions (e.g., to higher order (or fractional order) Sobolev spaces \cite{A,Martinazzi}, to unbounded domain in $\R^n$ \cite{AT00,R,LR,LL}, to singular weighted case \cite{AdiSan,AY,Nguyen} or to Riemannian manifolds \cite{Yang,YSK,MS,AT,MST,NguyenMT,NgoNguyen2016}). In $2004$, Adimurthi and Druet improved the Moser--Trudinger inequality \eqref{eq:classicalMT} in dimension $2$ by replacing the integral $\int_\Om e^{\al |u|^{\frac n{n-1}}} dx \leq 1$ by $\int_\Om e^{\al (1+ \gamma \|u\|_{L^n(\Om)}^n)^{\frac1{n-1}}|u|^{\frac n{n-1}}} dx$ with $0\leq \gamma < \lam_1(\Om) := \inf\{\|\na u\|_{L^2(\Om)}^2\, :\, u\in H^1_0(\Om);\, \|u\|_{L^2(\Om)} =1\}$. A higher dimension version of the Moser--Trudinger inequality is spirit of Adimurthi and Druet was established by Yang \cite{Yang06}. Tintarev \cite{Tin} improve the inequality of Adimurthi and Druet (but still in dimension $2$) by replacing the condition $\|\na u\|_{L^2(\Om)} \leq 1$ by a weaker condition $\|\na u\|_{L^2(\Om)}^2 -\gamma \|u\|_{L^2(\Om)}^2 \leq 1$ with $0\leq \gamma < \lam_1(\Om)$. In \cite{NguyenT,NguyenMT}, the author extends the result of Tintarev to the higher dimension as well as to the case of the singular--Moser--Trudinger inequality, respectively. Among these generalization of the Moser--Trudinger inequality \eqref{eq:classicalMT}, let us quote the singular Moser--Trudinger inequlity due to Adimurthi and Sandeep \cite{AdiSan}: for any bounded domain $\Om \subset \R^n$ containing the origin in its interior and $\beta \in [0,n)$, it holds 
\begin{equation}\label{eq:classicalSMT}
\sup_{u \in W_0^{1,n}(\Om), \|\na u\|_{L^n(\Om)} \leq 1} \int_\Om e^{\al(1-\frac\be n) |u|^{\frac n{n-1}}} dx < \infty,
\end{equation}
if and only if $\al \leq \al_n: = n \om_{n-1}^{\frac1{n-1}}$.

Another important inequality in the unit ball is the Hardy inequality which asserts that
\begin{equation}\label{eq:HardyB}
\int_{\B^n} |\na u|^n dx \geq \lt(\frac{2(n-1)}n\rt)^n \int_{\B^n} \frac{|u|^n}{(1 -|x|^2)^n} dx,\quad u\in C_0^\infty(\B^n).
\end{equation}
The constant $(\frac{2(n-1)}n)^n$ is sharp and never attained. Furthermore, it was proved by Mancini, Sandeep and Tintarev (see \cite[Lemma $2.1$]{MST}) that for any $p \in (n, \infty)$ there exists a constant $S_{n,p} >0$ depending only on $n$ and $p$ such that
\begin{equation}\label{eq:HardySobolevB}
\int_{\B^n} |\na u|^n dx - \lt(\frac{2(n-1)}n\rt)^n \int_{\B^n} \frac{|u|^n}{(1 -|x|^2)^n} dx \geq S_{n,p} \lt(\int_{\B^n} \frac{|u|^p}{(1 -|x|^2)^n} dx\rt)^{\frac np},\quad u \in C_0^\infty(\B^n).
\end{equation}
Denote
\[
\mathcal H(u) = \int_{\B^n} |\na u|^n dx - \lt(\frac{2(n-1)}n\rt)^n \int_{\B^n} \frac{|u|^n}{(1 -|x|^2)^n} dx,\quad u \in C_0^\infty(\B^n).
\]
By \eqref{eq:HardySobolevB}, the functional $u \to \sqrt{\mathcal H(u)}$ defines a norm on $C_0^\infty(\B^2)$. Let $\mathcal H(\B^2)$ denote the completion of $C_0^\infty(\B^2)$ under this norm. We have that $H_0^1(\B^2)$ is a proper subspace of $\mathcal H(\B^2)$. Notice that when $n\geq 3$ we do know whether or not the functional $u \to \sqrt[n]{\mathcal H(u)}$ defines a norm on $C_0^\infty(\B^n)$. However, by the density, $\mathcal H$ is well-defined on $W_0^{1,n}(\B^n)$.

In \cite{WangYe}, Wang and Ye obtained another improvement of the Moser--Trudinger inequality \eqref{eq:classicalMT} in the unit disc $\B^2\subset \R^2$ which combines both the Moser--Trudinger inequality  \eqref{eq:classicalMT} and the Hardy inequality \eqref{eq:Hardy}. Their inequality states that
\begin{equation}\label{eq:WangYe}
\sup_{u\in \mathcal H(\B^2), \mathcal H(u) \leq 1} \int_{\B^2} e^{4\pi u^2} dx < \infty.
\end{equation}
The proof of \eqref{eq:WangYe} given in \cite{WangYe} is based on the blow-up analysis method which is now a standard method to study the problems of this type. We refer the readers to \cite{Li2001,Li2005,NguyenT,Nguyen,R,LR,AD,WangYe,W19b,YZ} and references therein for more details on this method. The Hardy--Moser--Trudinger inequality \eqref{eq:WangYe} is a special case of the inequality of Tintarev \cite{Tin} aforementioned in which $\mathcal H$ is replaced by the functional $\mathcal H_V(u) = \int_{\B^2} |\na u|^2dx - \int_{\B^2} V u^2 dx$ for some potential $V$ so that $\mathcal H_V$ satisfies a weak coercive condition. There have been a lot of generalizations of \eqref{eq:WangYe} (see \cite{YZ,LY,LLY,W19a,W19b,Hou}). It is very remarkable that the inequality \eqref{eq:WangYe} can be seen as the analogue of the Hardy--Sobolev--Maz'ya inequality in dimension two. Recall that the Hardy--Sobolev --Maz'ya inequality (see \cite[Section $2.1.6$, Corollary $3$]{Maz'ya}) says that there exists a constant $C >0$ such that for any $u \in W^{1,2}_0(\B^n)$ with $n >2$, it holds
\[
\int_{\B^n} |\na u|^2 dx - \int_{\B^n} \frac{|u|^2}{(1-|x|^2)^2} dx \geq C \lt(\int_{\B^n} |u|^{\frac {2n}{n-2}} dx\rt)^{\frac{n-2}n}.
\]
Moreover, let $C_n$ denote the best constant so that the above inequality holds. It is well known that $C_n < S_n$ and is attained if $n>4$ (see \cite{TT}), and $C_3 = S_3$ and is not attained (see \cite{BFL}) where $S_n$ is the best constant in the Sobolev inequality \eqref{eq:Sobolev} with $p=2$ and $q =2^*$. The $L_p$ version of the above inequality in the hyperbolic space was considered in \cite{NguyenHS} by the author.

A new proof of \eqref{eq:WangYe} without using the blow-up analysis method was recently given by the author \cite{Nguyennew}. This new proof is based on the transplantation of Green functions. This method was previously used by Flucher \cite{F} to prove the existence of maximizer for the Moser--Trudinger inequality in dimension two, and then was used by Lin \cite{Lin1996} in any dimension. It also was successfully applied to prove the existence of maximizers for the singular Moser--Trudinger inequality \cite{CR15,CR16,CNR}. Let us explain briefly on this method. We know that $G_{\B^2}(x) = -\frac1{2\pi} \ln |x|$ is the Green function of $-\De$ in $\B^2$ with pole at $0$ and Dirichlet boundary condition. It was proved in \cite{WangYe} that the equation $-\De u - \frac1{(1-|x|^2)^2} u = \de_0$ in the distribution sense has a unique radial solution $G \in \mathcal H(\B^2) + W^{1,p}_0(B_{1/2})$ where $B_r =\{x \in \R^2\, :\, |x|< r\}$. The function $G$ is strictly decreasing and has the decomposition $G = -\frac1{2\pi} \ln |x| + C_G + \psi(x)$ for some constant $C_G$, where $\psi \in C_{loc}^{1,\al}(\B^2)$ and $\psi(x) = O(|x|^{1+ \al})$ as $x \to 0$ for any $\al \in (0,1)$. Back to the new proof of \eqref{eq:WangYe}, by the rearrangement argument applied to hyperbolic space, we only have to prove it for radial functions, i.e., the functions depend only on $|x|$. By abusing notation, we write $u(r)$ for the value of $u(x)$ with $|x| =r$ and a radial function $u$. For a radial function $u \in \mathcal H(\B^2)$, we define the new radial function $v$ on $\B^2$ such that $u(x) = v(e^{-2\pi G(x)})$. The main computations in \cite{Nguyennew} implies that $v \in W_0^{1,2}(\B^2)$ and $\|\na v\|_{L^2(\B^2)}^2 \leq \mathcal H(u)$, and 
\[
\int_{\B^2} e^{4\pi u^2} dx \leq e^{4\pi C_G} \int_{\B^2} e^{4\pi v^2}dx
\]
where $C_G$ appears in the decomposition of the Green function $G$ above. Then the inequality \eqref{eq:WangYe} follows from the classical Moser--Trudinger \eqref{eq:classicalMT} in $\B^2$.

The Moser--Trudinger inequality in the hyperbolic spaces was established by Mancini and Sandeep \cite{MS} (see also \cite{AT} by Adimurthi and Tintarev). In \cite{MST}, by using the inequality \eqref{eq:WangYe}, Mancini, Sandeep and Titarev have established the following Moser--Trudinger inequality in the hyperbolic spaces $\mathbb H^2$ 
\begin{equation}\label{eq:MTonhyperB2}
\sup_{u\in \mathcal H(\B^2), \mathcal H(u) \leq 1} \int_{\B^2} \frac{e^{4\pi u^2} -1 -4\pi u^2}{(1- |x|^2)^2} dx < \infty.
\end{equation}
In fact, it was show in in \cite{LY16} that the inequalities \eqref{eq:WangYe} and \eqref{eq:MTonhyperB2} are equivalent as well. The higher dimension version of \eqref{eq:MTonhyperB2} was conjectured in \cite{MST} (see the Conjecture $5.2$) as follows
\begin{equation}\label{eq:HMTconj}
\sup_{u\in C_0^\infty(\B^n), \mathcal H(u) \leq 1}\int_{\B^n} \frac{e^{\al_n |u|^{\frac n{n-1}}} -P_{n-1}(\al_n |u|^{\frac n{n-1}})}{(1 -|x|^2)^n} dx < \infty
\end{equation}
where $P_k(t) = e^t -\sum_{i=0}^{k} \frac{t^k}{k!}, t\geq 0, k\geq 0$. It also was shown in \cite{MST} that
\begin{equation*}
\sup_{u\in C_0^\infty(\B^n), \mathcal H(u) \leq 1}\int_{\B^n} \frac{e^{\al_n |u|^{\frac n{n-1}}} -P_{n-2}(\al_n |u|^{\frac n{n-1}})}{(1 -|x|^2)^n} dx =\infty.
\end{equation*}

The original motivation of this paper is to prove the conjectured inequality \eqref{eq:HMTconj}. In fact, we shall establish a singular Moser--Trudinger inequality in hyperbolic spaces which is more general than \eqref{eq:HMTconj} (see Theorem \ref{MST} below). In order to prove the conjectured inequality \eqref{eq:HMTconj}, we will prove the following singular Hardy--Moser--Trudinger inequality in the unit ball $\B^n$ which is the first main result in this paper.

\begin{theorem}\label{MAIN}
Let $n \geq 3$ and $0\leq \beta < n$, then there exists a constant $C(n,\beta)$ depending only on $n$ and $\beta$ such that
\begin{equation}\label{eq:SHMT}
\sup_{u\in W^{1,n}_0(\B^n), \mathcal H(u) \leq 1}\int_{\B^n} e^{(1-\frac\be n)\alpha_n |u|^{\frac n{n-1}}} |x|^{-\beta} dx \leq C(n,\beta).
\end{equation}
\end{theorem}
Obviously, the inequality \eqref{eq:SHMT} is stronger than the singular Moser--Trudinger inequality \eqref{eq:classicalSMT} in $\B^n$. Furthermore, it combines both the singular Moser--Trudinger inequality \eqref{eq:classicalSMT} and the Hardy inequality \eqref{eq:HardyB}. In the dimension two, the inequality \eqref{eq:SHMT} was recently proved by Wang \cite{W19b} by using the blow-up analysis method following the lines in the proof of Wang and Ye \cite{WangYe}. Our proof of \eqref{eq:SHMT} is completely different with their proofs. In fact, we follow the arguments in \cite{Nguyennew} in which the new proof of \eqref{eq:WangYe} is provided. The main feature in the proof is the existence of a Green function $G$ which is the weak solution of the equation 
\[
-\De_n G - \lt(\frac{2(n-1)}n\rt)^n \frac{G^{n-1}}{(1 -|x|^2)^n} = \de_0
\]
in $\B^n$ in the distribution sense, where $\De_n G = \text{\rm div} (|\na G|^{n-2} \na G)$ is the $n-$Laplace operator. The existence of $G$ follows from the deep results of Pinchover and Tintarev concerning to the $p-$Laplacian problems \cite{PT}. Some important properties of $G$ are given in Lemma \ref{Gfunction} below. It should be notice here that our approach can be applied to prove a more general class of the improvements of the singular Moser--Trudinger inequality \eqref{eq:classicalSMT} in $\B^n$ by replacing $\mathcal H(u)$ by $\mathcal H_V(u) = \|\na u\|_{L^n(\B^n)}^n - \int_{\B^n} V |u|^n dx$ with the potential $V$ satisfying some suitable condition. The details of this fact will be mentioned in the remark at the end of this paper.

As a consequence of Theorem \ref{MAIN}, we obtain the following singular Moser--Trudinger inequality in the hyperbolic spaces $\mathbb H^n$ which confirms affirmatively the inequality \eqref{eq:HMTconj} of Mancini, Sandeep and Tintarev.
\begin{theorem}\label{MST}
Let $n\geq 2$ and $0\leq \beta < n$, then there exists a constant $\tilde C(n,\beta)$ depending only on $n$ and $\beta$ such that
\begin{equation}\label{eq:MST}
\sup_{u\in W_0^{1,n}(\B^n), \mathcal H(u) \leq 1} \int_{\B^n} \frac{e^{(1-\frac\be n) \al_n |u|^{\frac n{n-1}}} - P_{n-1}\Big(\Big(1-\frac\be n\Big)\al_n |u|^{\frac n{n-1}}\Big)}{(1 -|x|^2)^n} |x|^{-\beta} dx < \tilde C(n,\beta).
\end{equation}

\end{theorem}
Evidently, when $\beta =0$, the inequality \eqref{eq:MST} is exactly the inequality \eqref{eq:HMTconj}. Theorem \ref{MST} hence not only confirms affirmatively the inequality \eqref{eq:HMTconj} but also extends this inequality to the singular case $0 < \beta < n$.

It is an interesting question on the extremal functions for the Moser--Trudinger inequality. The existence of extremals for the Moser--Trudinger inequality was first proved by Carleson and Chang \cite{CC} when $\Om = \B^n$ (Another proof of this result was given in \cite{DeFOR}). Later, this existence result was proved for any domain in $\R^2$ by Flucher \cite{F} and for any domain in $\R^n$ by Lin \cite{Lin1996}. Notice that the method used in \cite{F,Lin1996} is based on the transplantation of Green functions. This method was successfully applied in \cite{CR15,CR16,CNR} to prove the existence of extremals for the singular Moser--Trudinger inequality. For the improved Moser--Trudinger inequality, the existence of extremals was proved in \cite{NguyenT,Nguyen,YZsingular} and the references therein. In \cite{Li2001,Li2005}, Li developed a blow-up analysis method to establish the existence of extremals for the Moser--Trudinger inequality on Riemannian manifolds. Concerning to the Hardy--Moser--Trudinger inequality, it was proved by Wang and Ye \cite{WangYe} (by using the blow-up analysis method) that the extremals for the inequality \eqref{eq:WangYe} exists in $\mathcal H(\B^2)$ but not in $W^{1,2}_0(\B^2)$. Similarly, again by the the blow-up analysis method, Yang and Zhu proved the existence of extremals for the improvement version of \eqref{eq:WangYe} and Wang proved the existence of the singular Hardy--Moser--Trudinger inequality \eqref{eq:SHMT} in $\B^2$. It remains an open question in this paper which is whether or not the extremals for the singular Hardy--Moser--Trudinger inequality \eqref{eq:SHMT} exists when $n \geq 3$. The main difficult is to determine the suitable space for which the extremals (if exist) belong to. Let us recall that when $n \geq 3$, we do not know the functional $u \to \mathcal H(u)^{\frac1n}$ is a norm on $C_0^\infty(\B^n)$ or not. So we can't talk about the completion of $C_0^\infty(\B^n)$ under this functional also the weak convergence with respect to this functional. This is the crucial different with the case $n=2$. We will come back this question in the future research.

As a final remark, it is well known that for a convex domain domain $\Om \subset \R^n$ the following Hardy's inequality holds (see, e.g., \cite{Marcus,Mat})
\[
\int_\Om |\na u|^n dx \geq \lt(\frac{n-1}n\rt)^n \int_\Om \frac{|u|^n}{d(x,\pa \Om)^n} dx,\quad u\in C_0^\infty(\Om),
\]
where $d(x,\pa \Om) = \inf \{|x-y|\, :\, y\in \pa \Om$. The constant $(n-1)^n/n^n$ is sharp and never attained. Hence,
\[
H_\Om(u) = \int_\Om |\na u|^n dx - \lt(\frac{n-1}n\rt)^n \int_\Om \frac{|u|^n}{d(x,\pa \Om)^n} dx >0,\quad u \in C_0^\infty(\Om) \setminus\{0\}.
\]
We wonder if the inequality \eqref{eq:SHMT} can be extended to any convex domain $\Om$ in $\R^n$. In this direction, we propose the following inequality
\begin{equation}\label{eq:conj}
\sup_{u\in C_0^\infty(\Om), H_\Om(u) \leq 1} \int_\Om e^{\al_n(1-\frac\be n) |u|^{\frac n{n-1}}} |x|^{-\beta} dx < \infty.
\end{equation}
Since $d(x,\pa \B^n) = 1 -|x| \geq \frac{1-|x|^2}2$, then the inequality \eqref{eq:conj} holds when $\Om = \B^n$ by \eqref{eq:SHMT}. In dimension two, the inequality \eqref{eq:conj} for $\beta =0$ was conjectured by Wang and Ye (see \cite[Conjecture, page $4$]{WangYe}) and was recently settled by Lu and Yang \cite{LY16}.

The rest of this paper is organized as follows. In the section \S2 we recall some facts on the rearrangement arguments in the hyperbolic space which enables us reducing the proof of \eqref{eq:SHMT} to the radial functions in $\B^n$. We also prove the existence of the Green function $G$ and its properties in this section. Finally, we define a transformation of functions (based on the transplantation of Green functions) and make some useful computations which is useful in the proof of\eqref{eq:SHMT} in subsection \S2.3. The proofs of \eqref{eq:SHMT} and \eqref{eq:MST} are given in the section \S3. We also make some further comments on the application of our method to obtain the other improvements of the singular Moser--Trudinger inequality in $\B^n$ concerning to the potential $V$.

\section{Preliminaries}
In this section, we recall some useful facts and make some crucial estimates which will be used in the proof of Theorem \ref{MAIN}. We first recall the rearrangement argument applied to the hyperbolic spaces.
\subsection{Reducing to the radial functions}
In this subsection, we consider the hyperbolic space $\mathbb H^n$ as the unit ball $\B^n$ equipped with the Riemannian metric
\[
g(x) = \lt(\frac2{1-|x|^2}\rt)^2 (dx_1^2 + dx_2^2 + \cdots + d x_n^2).
\]
The volume element and the gradient operator with respect to $g$ is given by $\dvH = \lt(\frac 2{1-|x|^2}\rt)^n dx$ and $\na_g = (\frac{1 -|x|^2}2)^2 \na$. This model of hyperbolic space is especially useful for questions involving rotational symmetry. The geodesic distance between $x$ and $0$ is given by $\rho(x) = \ln \frac{1+|x|}{1 -|x|}$ and we denote by $B_{\mathbb H^n}(0,r)$ the geodesic ball in $\mathbb H^n$ with center at $0$ and radius $r$, i.e.,
\[
B_{\mathbb H^n}(0,r) = \{x\in \mathbb H^n\, :\, \rho(x) < r\}.
\]
For a measurable subset $A\in \mathbb H^n$, we use the notation $v_{\mathbb H^n}(A) = \int_A \dvH.$ Let $u$ be a measurable function in $\mathbb H^n$ such that $v_{\mathbb H^n}(\{x\, :\, |u(x)|> t\}) < \infty$ for any $t >0$. The non-increasing rearrangement function of $u$, denoted by $u^*$, is defined as
\[
u^*(x) = \inf\{s>0\, :\, v_{\mathbb H^n}(\{x\, :\, |u(x)| >s\}) \leq v_{\mathbb H^n}(B_{\mathbb H^n}(0,\rho(x)))\}.
\]
From the definition, we have $\int_{\B^n} (u^*)^n \dvH = \int_{\B^n} |u|^n \dvH$. The well-known P\'olya--Sz\"ego principle in hyperbolic spaces \cite{Ba} says that if $u \in W_0^{1,n}(\B^n)$ then $u^* \in W^{1,n}_0(\B^n)$ and 
\[
\int_{\B^n} |\na u^*|^n dx = \int_{\B^n} |\na_g u^*|_g^n \dvH \leq \int_{\B^n} |\na_g u|_g^n \dvH = \int_{\B^n} |\na u|^n dx.
\]
Thus, $\mathcal H (u^*) \leq \mathcal H(u)$.

Furthermore, by the Hardy--Littlewood inequality (see \cite{Brock}), we have
\begin{align*}
\int_{\B^n} e^{(1-\frac\be n)\alpha_n |u|^{\frac n{n-1}}} |x|^{-\beta} dx&=2^{-n} \int_{\B^n} e^{(1-\frac\be n)\alpha_n |u|^{\frac n{n-1}}} |x|^{-\beta}(1-|x|^2)^n  \dvH\\
&\leq 2^{-n}\int_{\B^n} e^{(1-\frac\be n)\alpha_n |u^*|^{\frac n{n-1}}} |x|^{-\beta}(1-|x|^2)^n  \dvH\\
&=\int_{\B^n} e^{(1-\frac\be n)\alpha_n |u^*|^{\frac n{n-1}}} |x|^{-\beta} dx
\end{align*}
by noticing that the rearrangement of $|x|^{-\beta}(1 - |x|^2)^n$ is just itself. Therefore we only need to consider nonincreasing, radially symmetric functions in proving \eqref{eq:SHMT}. Let us define
\begin{equation*}
\Sigma =\{u \in C_0^\infty(\B^n)\, :\, u(x) = u(r)\quad\text{\rm with}\quad |x| =r;\, u'\leq 0\},
\end{equation*}
and $\mathcal H_1$ be the closure of $\Sigma$ in $W_0^{1,n}(\B^n)$. So, to prove Theorem \ref{MAIN}, we need only to show that there exists some constant $C(n,\beta)$ depending only on $n$ and $\beta$ such that
\begin{equation*}
\sup_{u\in \Sigma, \mathcal H(u) \leq 1}\int_{\B^n} e^{(1-\frac\be n)\alpha_n |u|^{\frac n{n-1}}} |x|^{-\beta} dx \leq C(n,\beta).
\end{equation*}


\subsection{Existence of Green function and its properties}
Throughout this subsection, we denote by
\[
V(x) = \lt(\frac{2(n-1)}n\rt)^n \frac{1}{(1-|x|^2)^n},
\]
and $Q_V(u) = \mathcal H(u), u\in C_0^\infty(\B^n)$, i.e.,
\[
Q_V(u) = \int_{\B^n} |\na u|^n dx - \int_{\B^n} V(x) |u(x)|^n dx,\quad u\in C_0^\infty(\B^n).
\]
We have $Q_V \geq 0$ on $C_0^\infty(\B^n)$ by the Hardy inequality \eqref{eq:HardyB}. By \cite[Theorem $5.4$]{PT}, the equation $Q_V'(u) = 0$ has (up to a multiple constant) a unique positive solution $v$ in $\B^n\setminus \{0\}$ of minimal growth in a neighborhood of infinity in $\B^n$ (see \cite[Definition $5.3$]{PT} for the definition of positive solution of minimal growth in a neighborhood of infinity). Furthermore, $v$ is either a global minimal solution of the equation $Q_V'(u) =0$ in $\B^n$, or $v$ has a nonremovable singularity at $0$.

By the Hardy--Sobolev inequality \eqref{eq:HardySobolevB}, there exists a positive constant $C >0$ such that
\[
Q_V(u) \geq  C \int_{\B^n} |u|^n dx,\quad u\in C_0^\infty(\B^n).
\]
In terminology of \cite[Definition $1.3$]{PT}, the functional $Q_V$ has a weighted spectral gap in $\B^n$ (or $Q_V$ is strictly positive in $\B^n$). This fact together with \cite[Theorem $5.5$]{PT} implies that the solution $v$ of the equation $Q_V'(u) =0$ in $\B^n\setminus \{0\}$ above has a nonremovable singularity at $0$. By Lemma $5.1$ in \cite{PT}, we have
\[
\lim_{x\to 0} \frac{v(x)}{-\ln |x|} = C
\]
for some $C >0$. By normalizing, we assume this solution satisfies
\begin{equation}\label{eq:chuanhoa}
\lim_{x\to 0} \frac{v(x)}{-\ln |x|} = \om_{n-1}^{-\frac1{n-1}}.
\end{equation}
Let $G(x)$ denote such a solution $v$, and we call it the Green function of the equation $Q_V'(u) =0$ in $\B^n$ with a pole at $0$. It is not hard to see that $G$ is the weak solution of the equation
\[
-\De_n G - \lt(\frac{2(n-1)}n\rt)^n \frac{G^{n-1}}{(1-|x|^2)^n} = \de_0
\]
in the distribution sense in $\B^n$. We have the following results on $G$.

\begin{lemma}\label{Gfunction}
$G$ is radially symmetric and strictly decreasing in $|x|$. There exists $C >0$ such that 
\begin{equation}\label{eq:chanG}
G(x) \leq C (1- |x|^2)^{\frac{n-1}n},\quad \frac12 \leq |x| < 1.
\end{equation}
Furthermore, we have the following decomposition of $G$
\begin{equation}\label{eq:decompG}
G(x) = -\om_{n-1}^{-\frac 1{n-1}} \ln |x| + C_G + H(x),
\end{equation}
with $H \in C_{loc}^{1,\al}(B)$ and $H(r) = O(r^{1+ \al})$ as $r \to 0$ for any $\al \in (0,1)$.
\end{lemma}
\begin{proof}
Since $V \in C^\infty(\B^n)$ and $G \in W^{1,n}_{loc}(\B^n\setminus\{0\})$, then by the standard regularity \cite{Serrin,Tol} we have $G \in C^1(\B^n \setminus \{0\})$. For any $R \in O(n)$ the group of the $n\times n$ orthogonal matrices. Denote $G_R(x) = G(Rx), x \in \B^n\setminus \{0\}$. It is easy to check that $G_R$ is a solution of the equation $Q_V'(u) = 0$ in $\B^n\setminus\{0\}$ and satisfies \eqref{eq:chuanhoa}. Hence $G_R \equiv G$ by the uniqueness. In other word, we have $G(Rx) = G(x)$ for any $R \in O(n)$. This implies that $G$ is radially symmetric in $|x|$.

By \eqref{eq:chuanhoa}, we have $G \in L^p_{loc}(\B^n,\dvH)$ for any $p < \infty$. For any $0< a < b < 1$, we chose $k_0 >0$ such that $k_0 a \geq 1$ and $k_0(1-b) >1$. For any $k\geq k_0$, we define the function
\[
\psi_k(x) = \begin{cases} 
0 &\mbox{if $0\leq |x| < a-\frac1k$ or $b+ \frac1k \leq |x| < 1$}\\
1-k(a-|x|) &\mbox{if $a- \frac1k \leq |x| < a$}\\
1 &\mbox{if $a \leq |x| < b$}\\
1-k(|x| -b) &\mbox{if $b\leq |x| < b+ \frac1k$.}
\end{cases}
\]
Testing the equation $Q_V'(G) = 0$ by $\psi_k$ and using the radially symmetric of $G$, we have
\[
\om_{n-1} \lt(k \int_{a-\frac1k}^a |G'(r)|^{n-2} G'(r) r^{n-1} dr - k \int_b^{b+ \frac1k} |G'(r)|^{n-2} G'(r) r^{n-1} dr\rt) = \int_{\B^n} V G^{N-1} \psi_k dx.
\]
Letting $k \to \infty$ and using the facts $G \in C^1(\B^n \setminus\{0\})$ and $G \in L^p_{loc}(\B^n,\dvH)$ for any $p< \infty$ and using the Lebesgue dominated convergence theorem, we get
\begin{equation}\label{eq:daoham1}
\om_{n-1} \lt(|G'(a)|^{N-2} G'(a) a^{N-1} - |G'(b)|^{N-2} G'(b) b^{N-1} \rt) = \int_{\{a \leq |x| < b\}} V G^{n-1} dx.
\end{equation}
From \eqref{eq:chuanhoa}, there exists a sequence $a_i \in (0,1)$ such that $a_i \to 0$ as $i \to \infty$ and $G'(a_i) < 0$. This fact together with \eqref{eq:daoham1} implies $G'(r) < 0$ for any $0< r < 1$. Hence $G$ is strictly decreasing in $|x|$.

We next prove \eqref{eq:chanG}. Let $B_{r} = \{x\, :\, |x| < r\}$ for $0< r < 1$. From the proof of Theorem $5.4$ in \cite{PT}, we see that $G$ is locally uniform limit in $\B^n \setminus\{0\}$ of the sequence $G_N$, $N\geq 2$ which solves the equation $Q_V'(G_N) = 0$ in $B_{1-\frac1N} \setminus \{0\}$ and satisfies the condition $G_N = 0$ on $\pa B_{1-\frac1N}$ and 
\[
\lim_{x\to 0} \frac{G_N(x)}{-\ln |x|} = \om_{n-1}^{-\frac1{n-1}}.
\]
Fix a $\de \in (0,1)$. Evidently, $G_N (x) \leq C_\de$ on $\pa B_\de$ for any $N\geq 2$ and for some $C_\de >0$ depending only on $\de$. Let $\psi(x) = (-\ln |x|)^{\frac{n-1}n}$. By a direct computation, we have
\[
-\De_n \psi (x) - V(x) \psi(x)^{\frac{n-1}n} =\lt(\frac{n-1}n\rt)^n \frac{\psi(x)^{n-1}}{|x|^n} \lt(\frac1{(-\ln |x|)^n }- \lt(\frac{2|x|}{1 -|x|^2}\rt)^n\rt).
\]
Using the elementary inequality
\[
-2r \ln r \leq 1 -r^2,\quad r\in (0,1),
\]
we obtain that 
\[
-\De_n \psi (x) - V(x) \psi(x)^{\frac{n-1}n} > 0,\quad x \in \B^n\setminus \{0\}.
\]
Notice that $\psi >0$ on $\pa B_{1-\frac1N}$. Furthermore, multiplying $\psi$ by a large constant $C$, we see that $C \psi \geq C_{1/2} \geq G_N$ on $\pa B_{1/2}$ for any $N$. Applying the comparison principle (see \cite[Theorem $2.2$]{PT} or \cite[Theorem $5$]{GS}), we have $G_N(x) \leq C \psi(x)$ for any $N$ and $\frac12 \leq |x| < 1$. Letting $N\to \infty$ we have 
\[
G(x) \leq C (-\ln |x|)^{\frac{n-1}n}  \leq \tilde C (1-|x|^2)^{\frac{n-1}n},\quad \frac12\leq |x| < 1.
\]
as wanted.

From \eqref{eq:daoham1}, we see that there exists 
\begin{equation}\label{eq:daoham2}
\lim_{a\to 0} |G'(a)|^{N-2} G'(a) a^{n-1} = |G'(b)|^{n-2} G'(b) b^{n-1} + \om_{n-1}^{-1} \int_{B_b} V G^{n-1} dx.
\end{equation}
Notice that $G' < 0$, hence there exists the limit
\[
\lim_{r\to 0} -G'(r) r = \gamma \geq 0.
\]
This limit together with \eqref{eq:chuanhoa} and L'H\^opital theorem implies $\ga = \om_{n-1}^{-\frac1{n-1}}$. Furthermore, we have from \eqref{eq:daoham2}
\[
|G'(b)|^{n-2} G'(b) b^{n-1} + \om_{n-1}^{-1} \int_{B_b} V G^{n-1} dx = -\ga^{n-1}, \quad\forall\, 0< b < 1,
\]
or equivalently,
\begin{equation}\label{eq:Gchan}
-\om_{n-1}^{\frac1{n-1}} G'(b) b = \lt(1+ \int_{B_b} V G^{n-1} dx\rt)^{\frac1{n-1}},\quad\forall\, 0< b < 1.
\end{equation}
Again, from \eqref{eq:daoham1}, we get
\[
-G'(r) r = \lt(\ga^{n-1} +\om_{n-1}^{-\frac1{n-1}} \int_{B_r} V G^{n-1} dx\rt)^{\frac1{n-1}} = \gamma + \psi(r),
\]
with
\[
\psi(r) = \lt(\ga^{n-1} +\om_{n-1}^{-\frac1{n-1}} \int_{B_r} V G^{n-1} dx\rt)^{\frac1{n-1}} - \gamma.
\]
From \eqref{eq:chuanhoa}, we have $\psi(r) = O((-\ln r)^{n-1} r^n)$ as $r\to 0$. Furthermore, we have $\psi \in C_{loc}^{1,\al}(B)$ for any $\al \in (0,1)$. Now, we have
\begin{equation}\label{eq:daoham3}
-G'(r) -\frac{\gamma}{r} = \frac{\psi(r)}r = O((-\ln r)^{n-1} r^{n-1})
\end{equation}
as $r \to 0$ which implies for any $0< s < r$
\[
|-G(r) - \gamma \ln (r) - (-G(s) -\gamma \ln s)| = \int_s^r \frac{\psi(t)}{t} dt \to 0
\]
as $r,s \to 0$. Hence, there exits the limits $\lim_{r\to 0} (-G(r) - \gamma \ln (r)) = -C_G$. Hence, we get from \eqref{eq:daoham3} that
\[
-G(r) -\gamma \ln r + C_G = \int_0^r \frac{\psi(s)}s ds.
\]
Let $H(r) = -\int_0^r \frac{\psi(s)}s ds$, we obtain
\[
G(r) = -\om_{n-1}^{-\frac1{n-1}} \ln r + C_G + H(r)
\]
by noticing that $\gamma = \om_{n-1}^{-\frac1{n-1}}$. From the definition of $H$, we have $H(r) =O((-\ln r)^{n-1} r^n)$ as $r \to 0$ and $H \in C_{loc}^{1,\al} (B)$ for any $\al \in (0,1)$.
\end{proof}

\subsection{A transformation of functions via the transplantation of Green functions}
Let us recall that the $n$-Green function with pole at $0$ of the operator $-\De_n$ in $B$ is given
\[
G_{\B^n}(x) = -\om_{n-1}^{-\frac1{n-1}} \ln |x|,
\]
i.e., $G_{\B^n}$ is the weak solution of the equation $-\De_n G_{\B^n} = \de_0$ in $\B^n$ and $G_{\B^n} =0$ on $\pa B$.

Let $u \in \Sigma$ be a given function and we define a new function $v$ in $\B^n$ by
\begin{equation}\label{eq:v}
v(r) = u(G^{-1}\circ G_{\B^n}(r))
\end{equation}
or equivalently
\[
u(r) = v(e^{-\om_{n-1}^{\frac1{n-1}} G(r)}).
\]
A simple computation shows
\[
u'(r) = -v'(e^{-\om_{n-1}^{\frac1{n-1}} G(r)})) e^{-\om_{n-1}^{\frac1{n-1}} G(r)}\om_{n-1}^{\frac1{n-1}} G'(r).
\]
Thus, we have
\begin{align*}
\int_{\B^n} |\na u|^n dx& =\om_{n-1}\int_0^1 |u'(r)|^n r^{n-1} dr\\
&= \om_{n-1}\int_0^1 |v'(e^{-\om_{n-1}^{\frac1{n-1}} G(r)})) e^{-\om_{n-1}^{\frac1{n-1}} G(r)}\om_{n-1}^{\frac1{n-1}} G'(r)|^n r^{n-1} dr.
\end{align*}
Making the change of variable $t =e^{-\om_{n-1}^{\frac1{n-1}} G(r)} $ and define
\[
a(t) = G^{-1}(-\om_{n-1}^{-\frac1{n-1}} \ln t).
\]
Since $G$ is strictly decreasing, then $a$ is strictly increasing, $a(0) =0$ and $a(1) =1$. Furthermore, $a \in C^1([0,1))$. From the change of variable above, we have $r = a(t)$ and $dr = -(G^{-1})'(-\om_{n-1}^{-\frac1{n-1}} \ln t)\om_{n-1}^{-\frac1{n-1}} t^{-1} dt$ and
\begin{align}\label{eq:grad}
\int_{\B^n} |\na u|^n dx&=-\om_{n-1}\int_0^1 |v'(t)|^n t^n \om_{n-1}^{\frac n{n-1}} |G'(a(t))|^n a(t)^{n-1}(G^{-1})'(-\om_{n-1}^{-\frac1{n-1}} \ln t) \om_{n-1}^{-\frac1{n-1}} t^{-1} dt\notag\\
&=\om_{n-1}\int_0^1 |v'(t)|^n t^{n-1} \om_{n-1} |G'(a(t))|^{n-1} a(t)^{n-1} dt\notag\\
&= \om_{n-1} \int_0^1 |v'(t)|^n t^{n-1} dt + \om_{n-1}\int_0^1 |v'(t)|^n t^{n-1} \Phi(t) dt,
\end{align}
with 
\[
\Phi(t) = \om_{n-1} |G'(a(t))|^{n-1} a(t)^{n-1} -1 >0,\quad \forall \, t\in (0,1)
\]
by \eqref{eq:Gchan}, here we used the equality $G'(G^{-1}(a)) (G^{-1})'(a) =1$ for the second equality. Note that
\begin{align*}
\Phi'(t) &= -\om_{n-1}\De_n G(a(t))\, a(t)^{n-1} a'(t)\notag\\
&=\lt(\frac{2(n-1)}n\rt)^n (-\ln t)^{n-1} \frac{a(t)^{n-1}}{(1 -a(t)^2)^n} \frac{1}{-\om_{n-1}^{\frac1{n-1}} G'(a(t)) t}
\end{align*}

In the other hand
\begin{align}\label{eq:u}
\lt(\frac{2(n-1)}n\rt)^n\int_{\B^n} \frac{|u|^n}{(1- |x|^2)^n} dx & =\lt(\frac{2(n-1)}n\rt)^n\om_{n-1} \int_0^1 \frac{|v(e^{-\om_{n-1}^{\frac1{n-1}} G(r)})|^n}{(1 -r^2)^n} r^{n-1} dr\notag\\
&= \lt(\frac{2(n-1)}n\rt)^n\om_{n-1} \int_0^1 v(t)^n \frac{a(t)^{n-1}}{(1 -a(t)^2)^n} \frac{dt}{-\om_{n-1}^{\frac1{n-1}} G'(a(t)) t}\notag\\
&=  \om_{n-1}\int_0^1 v(t)^n \frac{\Phi'(t)}{(-\ln t)^{n-1}} dt.
\end{align}

To continue, we need a Hardy type inequality as follows
\begin{lemma}
For any $v \in C_0^1([0,1))$ which is non-increasing, it holds
\begin{equation}\label{eq:Hardy}
\int_0^1 |v'(t)|^n t^{n-1} \Phi(t) dt \geq \int_0^1 v(t)^n \frac{\Phi'(t)}{(-\ln t)^{n-1}} dt.
\end{equation}
\end{lemma}
\begin{proof}
Let $v(t) = w(t) (-\ln t)$. We have $v'(t) = w'(t) (-\ln t) -\frac{w(t)} t$. Notice that $v'(t) \leq 0$. Using the simple inequality
\[
|a-b|^n \geq |b|^n - n b^{n-1} a +  |a|^n,
\]
for any $b\geq 0$ and $a - b \leq 0$, we get
\[
|v'(t)|^n \geq \frac{w(t)^n}{t^n} + n \frac{w(t)^{n-1} w'(t)}{t^{n-1}} \ln t + |w'(t)|^n (-\ln t)^n = (w(t)^n \ln t)' t^{1-n} + |w'(t)|^n (-\ln t)^n.
\]
Using integration by parts, we get
\begin{align*}
\int_0^1 |v'(t)|^n t^{n-1} \Phi(t) dt&\geq \int_0^1 (w(t)^n \ln t)' \Phi(t) dt + \int_0^1 |w'(t)|^n (-\ln t)^n t^{n-1} \Phi(t) dt\\
&\geq \int_0^1 w(t)^n (-\ln t) \Phi'(t) dt\\
&= \int_0^1 v(t)^n \frac{\Phi'(t)}{(-\ln t)^{n-1}} dt,
\end{align*}
as desired.
\end{proof}

Combining \eqref{eq:grad}, \eqref{eq:u} and \eqref{eq:Hardy}, we arrive
\begin{equation}\label{eq:key}
\mathcal H (u)^n \geq \int_{\B^n} |\na v|^n dx.
\end{equation}
The inequality \eqref{eq:key} is the key in our proof of Theorem \ref{MAIN}.

\section{Proof of Theorem \ref{MAIN} and Theorem \ref{MST}}
With the estimate \eqref{eq:key} at hand, we are ready to prove the inequality \eqref{eq:SHMT} in Theorem \ref{MAIN}.
\begin{proof}[Proof of Theorem \ref{MAIN}]
As mentioned in subsection \S2.1, it is enough to prove Theorem \ref{MAIN} for function in $\Sigma$. For any $u \in \Sigma$ with $\mathcal H(u) \leq 1$, we define the new function $v$ by \eqref{eq:v}. Notice that $v \in W^{1,n}_0(\B^n)$ and by \eqref{eq:key} we have $\int_{\B^n} |\na v|^n dx \leq 1$.

Moreover, by the simple calculations, we have
\begin{align*}
\int_{\B^n} e^{(1-\frac \beta n)\al_n |u|^{\frac n{n-1}}} |x|^{-\beta} dx &= \om_{n-1} \int_0^1 e^{(1-\frac\beta n)\al_n v(e^{-\om_{n-1}^{\frac1{n-1}} G(r)}))^{\frac n{n-1}}} r^{n-\beta -1} dr\\
&=\om_{n-1} \int_0^1 e^{(1-\frac\be n)\al_n v(t)^{\frac n{n-1}}} t^{n-\beta -1}\lt(\frac{a(t)}{t}\rt)^{n-\beta} \frac{1}{-\om_{n-1}^{\frac 1{n-1}} G'(a(t)) a(t)} dt\\
&= \om_{n-1} \int_0^1 e^{\al_n v(t)^{\frac n{n-1}}} t^{n-\beta -1} \Psi(t) dt,
\end{align*}
with
\[
\Psi(t) =  \frac{1}{-\om_{n-1}^{\frac 1{n-1}} G'(a(t)) a(t)}\frac{a(t)^{n-\beta}}{t^{n-\beta}}.
\]
By \eqref{eq:Gchan}, we have
\[
-\om_{n-1}^{\frac 1{n-1}}G'(a(t)) a(t) > 1,\quad\forall \, t\in (0,1).
\]
From the definition of $a(t)$, we have $G'(a(t)) a'(t) = -\om_{n-1}^{-\frac1{n-1}} t^{-1}$, hence
\[
\lt(\frac{a(t)}t\rt)' = \frac{a'(t) t - a(t)}t^2 = \frac{-1 - \om_{n-1}^{\frac1{n-1}}G'(a(t)) a(t)}{\om_{n-1}^{\frac1{n-1}}t^2 G'(a(t)} < 0,\quad\forall \, t\in (0,1)
\]
since $G' < 0$. Then the function $a(t)/t$ is strictly decreasing. Furthermore, from \eqref{eq:decompG} we have 
\[
G(r) = -\om_{n-1}^{-\frac 1{n-1}} \ln r + C_G + H(r)
\] 
which implies
\[
\frac{a(t)} {t} = e^{\om_{n-1}^{\frac1{n-1}} (C_G + H(a(t)))}.
\]
So, we have
\[
\frac{a(t)}{t} < \lim_{t\to 0} \frac{a(t)}{t} = e^{\om_{n-1}^{\frac1{n-1}} C_G},\quad\forall \, t\in (0,1)
\]
since $\lim_{t\to 0} a(t) =0$. Therefore,  it holds
\[
\int_{\B^n} e^{(1-\frac \beta n)\al_n |u|^{\frac n{n-1}}} |x|^{-\beta} dx \leq e^{(1-\frac\be n) \al_n C_G}  \int_{\B^n} e^{(1-\frac\be n) |v|^{\frac n{n-1}}} |x|^{-\beta} dx.
\]
In the light of the classical singular Moser--Trudinger inequality \eqref{eq:classicalSMT} in $\B^n$, it holds
\begin{align*}
\int_{\B^n} e^{(1-\frac \beta n)\al_n |u|^{\frac n{n-1}}} |x|^{-\beta} dx &\leq e^{(1-\frac\be n) \al_n C_G}  \sup_{w \in W^{1,n}_0(B), \int_{\B^n} |\na w|^n dx \leq 1}\int_{\B^n} e^{(1-\frac\be n)\al_n |w|^{\frac n{n-1}}} |x|^{-\beta} dx\\
&=:C(n,\beta)\\
&< \infty,
\end{align*}
for any $u\in \Sigma$ with $\mathcal H(u) \leq 1$. This finishes the proof of Theorem \ref{MAIN}.

\end{proof}

We next prove Theorem \ref{MST}.

\begin{proof}[Proof of Theorem \ref{MST}]
Again, by the standard rearrangement argument in the hyperbolic spaces from subsection \S2.1, it is enough to prove the inequality \eqref{eq:MST} for function $u \in \Sigma$ with $\mathcal H(u) \leq 1$. For such a function $u$, we have
\[
u(r) \leq C_{n,p} (1-r^2)^{\frac{n-1}p},\quad \forall\, r\in (1/2,1),
\]
here $p > n$ is any number and $C_{n,p}$ depends only on $n$ and $p$ (see \cite[Lemma $5.3$]{MST}). Hence, 
\[
e^{(1-\frac\be n) \al_n |u|^{\frac n{n-1}}} - P_{n-1}\Big(\Big(1-\frac\be n\Big)\al_n |u|^{\frac n{n-1}}\Big) \leq \tilde C_{n,p,\beta} (1-r^2)^{\frac{n^2}p}, \quad \forall\, r\in (1/2,1),
\]
here $p > n$ is any number and $\tilde C_{n,p,\be}$ depends only on $n, p$ and $\be$. Choosing $p$ such that $n < p < \frac{n^2}{n-1}$ hence $\frac{n^2}p - n + 1 >0$. By splitting the integral, we have
\begin{align*}
&\int_{\B^n} \frac{e^{(1-\frac\be n) \al_n |u|^{\frac n{n-1}}} - P_{n-1}\Big(\Big(1-\frac\be n\Big)\al_n |u|^{\frac n{n-1}}\Big)}{(1 -|x|^2)^n} |x|^{-\beta} dx\\
& \qquad= \int_{\{|x| \leq \frac12\}} \frac{e^{(1-\frac\be n) \al_n |u|^{\frac n{n-1}}} - P_{n-1}\Big(\Big(1-\frac\be n\Big)\al_n |u|^{\frac n{n-1}}\Big)}{(1 -|x|^2)^n} |x|^{-\beta} dx  \\
&\qquad\quad + \int_{\{\frac12 < |x| < 1\}} \frac{e^{(1-\frac\be n) \al_n |u|^{\frac n{n-1}}} - P_{n-1}\Big(\Big(1-\frac\be n\Big)\al_n |u|^{\frac n{n-1}}\Big)}{(1 -|x|^2)^n} |x|^{-\beta} dx\\
&\qquad\leq \Big(\frac 43\Big)^n \int_{\{|x| \leq \frac12\}} e^{(1-\frac\be n) \al_n |u|^{\frac n{n-1}}} |x|^{-\beta} dx \\
&\qquad \quad + \tilde C_{n,p,\beta} \int_{\{\frac12 < |x| < 1\}} (1- |x|^2)^{\frac{n^2}p -n} dx \\
&\qquad \leq \Big(\frac 43\Big)^n \int_{\B^n} e^{(1-\frac\be n) \al_n |u|^{\frac n{n-1}}} |x|^{-\beta} dx \\
&\qquad\quad + \tilde C_{n,p,\beta} \int_{\{\frac12 < |x| < 1\}} (1- |x|^2)^{\frac{n^2}p -n} dx\\
&\qquad =:\tilde C(n,\beta)\\
&\qquad < \infty,
\end{align*}
here we use the inequality \eqref{eq:SHMT} in Theorem \ref{MAIN} and the fact $\frac{n^2}p -n + 1 >0$. This completes the proof of Theorem \ref{MST}.
\end{proof}

Finally, we make some further comments on our approach in this paper to the other improvement of the singular Moser--Trudinger inequality \eqref{eq:classicalSMT} in $\B^n$. Let $V:\B^n \to (0,\infty)$ be a radially symmetric, continuous potential such that $(1-|x|^2)^n V(x)$ is non-increasing in $|x|$ (this assumption enables us to apply the rearrangement argument in the hyperbolic spaces). We further assume that the functional
\[
Q_V(u) = \int_{\B^n} |\na u|^n dx - \int_{\B^n} V |u|^n dx 
\]
has a spectral gap (or strictly positive) in $\B^n$ in the sense of \cite[Definition $1.3$]{PT}. As in subsection \S2.2, we can prove that the equation
\[
-\De_n u - V u^{n-1} = \de_0,
\]
in the distribution sense in $\B^n$ has a unique radially symmetric, strictly decreasing (in $|x|$), positive solution $G$. Hence, there exists $a = \lim_{r\to 1} G(r) \geq 0$. We show that $a =0$. Indeed, from the proof of Theorem $5.4$ in \cite{PT}, we see that $G$ is locally uniform limit in $\B^n \setminus\{0\}$ of the sequence $G_N$ which solves the equation $Q_V'(G_N) = 0$ in $B_{1-\frac1N} \setminus \{0\}$ and satisfies the condition $G_N = 0$ on $\pa B_{1-\frac1N}$ and 
\[
\lim_{x\to 0} \frac{G_N(x)}{-\ln |x|} = \om_{n-1}^{-\frac1{n-1}}.
\]
If $a >0$, denote $H(x) = G(x) -a$. We have
\[
-\Delta_n H - V H^{n-1} = -\Delta_n G - V H^{n-1} = V(G^{n-1} -H^{n-1}) \geq 0
\]
in $\{x\, :\, \frac12 < |x| < 1\}$. Notice that $H >0$ in $\B^n$ and $G_N$ is uniformly bounded in $\pa B_{\frac12}$ so we can apply the comparison principle (see \cite[Theorem $2.2$]{PT} or \cite[Theorem $5$]{GS}) to get that $G_N(x) \leq C H(x)$ for any $\frac12 \leq |x| <1$, and any $N \geq 2$ and for some constant $C \geq 1$. Letting $N\to \infty$ we get $G(x) \leq C H(x)$ for some $C \geq 1$. Letting $|x|\to 1$ we obtain $a =0$. Therefore, the function $G: (0,1) \to (0,\infty)$ is a bijection. Furthermore, by the same arguments in subsection \S2.3, we can prove that the function $G$ has the form
\[
G(x) = -\om_{n-1}^{-\frac1{n-1}} \ln |x| + C_G + \psi(x)
\]
with $\psi \in C_{loc}^{1,\al}(\B^n)$ and $\psi(r) = O(r^{1+\al})$ as $r\to 0$ for any $\al \in (0,1)$. Now, we can follow the proof of Theorem \ref{MAIN} to prove the following inequality
\begin{equation}\label{eq:general}
\sup_{u\in W_0^1(\B^n), Q_V(u) \leq 1} \int_{\B^n} e^{(1-\frac\beta n) \al_n |u|^{\frac n{n-1}}} |x|^{-\beta} dx < \infty,
\end{equation}
for any $0\leq \beta < n$. In dimension two, the inequality \eqref{eq:general} was considered by Tintarev \cite{Tin} when $\beta =0$. A special example of the potential $V$ which satisfies our assumptions in $V(x) = \alpha \in [0,\lam_{1,n}(\B^n))$ where
\[
\lam_{1,n}(\B^n) = \inf\lt\{\int_{\B^n} |\na u|^n dx\, :\, u\in W^{1,n}_0(\B^n); \, \int_{\B^n} |u|^n dx =1\rt\}.
\]
In this case, we obtain the results in \cite{Nguyen,NguyenT} from \eqref{eq:general}
\[
\sup_{u\in W^{1,n}_0(\B^n), \|\na u\|_{L^n(\B^n)}^n - \alpha \|u\|_{L^n(\B^n)}^n \leq 1} \int_{\B^n} e^{\al_n(1-\frac\be n) |u|^{\frac{n}{n-1}}} |x|^{-\beta} dx < \infty,\quad \beta \in [0,n).
\]
Another example is the improvement of the singular Hardy--Moser--Trudinger inequality \eqref{eq:SHMT}. Let
\[
\lam_1 = \inf\lt\{\mathcal H(u)\,:\, u\in C_0^\infty(\B^n); \, \int_{\B^n} |u|^n dx =1\rt\}.
\]
The Poincar\'e--Sobolev inequality \eqref{eq:HardySobolevB} implies that $\lam_1 >0$. Therefore, for any $\lam \in [0,\lam_1)$ the potential $V(x) = \lt(\frac{2(n-1)}n\rt)^n (1-|x|^2)^{-n} + \lam$ satisfies our assumptions. Hence, we obtain the following improvement of the singular Hardy--Moser--Trudinger inequality \eqref{eq:SHMT}
\begin{equation}\label{eq:improvedSHMT}
\sup_{u\in W^{1,n}_0(\B^n), \mathcal H(u) -\lam \|u\|_{L^n(\B^n)}^n \leq 1} \int_{\B^n} e^{\al_n(1-\frac\be n) |u|^{\frac{n}{n-1}}} |x|^{-\beta} dx < \infty,\quad \beta \in [0,n).
\end{equation}
In dimension two, the inequality \eqref{eq:improvedSHMT} was established by Yang and Zhu \cite{YZ} for $\beta =0$ and by Hou \cite{Hou} for $\beta \in (0,n)$ by exploiting the blow-up analysis method. Following the proof of Theorem \ref{MST} and using the inequality \eqref{eq:improvedSHMT}, we obtain the following improvement of \eqref{eq:MST}: for any $\lam \in [0,\lam_1)$ and $\beta \in [0,n)$ it holds
\[
\sup_{u\in W_0^{1,n}(\B^n), \mathcal H(u)-\lam \|u\|_{L^n(\B^n)}^n \leq 1} \int_{\B^n} \frac{e^{(1-\frac\be n) \al_n |u|^{\frac n{n-1}}} - P_{n-1}\Big(\Big(1-\frac\be n\Big)\al_n |u|^{\frac n{n-1}}\Big)}{(1 -|x|^2)^n} |x|^{-\beta} dx < \infty.
\]



\begin{thebibliography}{9999}
\bibitem{AT00}
S. Adachi, and K. Tanaka, \emph{Trudinger type inequalities in $\R^N$ and their best exponents\text}, Proc. Amer. Math. Soc., {\bf 128} (2000), no. 7, 2051--2057. 

\bibitem{A}
D. R. Adams, \emph{A sharp inequality of J. Moser for higher order derivatives\text}, Ann. of Math., {\bf 128} (2) (1988) 385-398.

\bibitem{AD}
Adimurthi, and O. Druet, \emph{Blow--up analysis in dimension $2$ and a sharp form of Trudinger--Moser inequality\text}, Comm. Partial Differential Equations, {\bf 29} (2004), no. 1-2, 295--322.

\bibitem{AdiSan}
Adimurthi, and K. Sandeep, \emph{A singular Moser--Trudinger embedding and its applications\text}, NoDEA Nonlinear Differential Equations Appl., {\bf 13} (2007), no. 5-6, 585--603.

\bibitem{AT}
Adimurthi, and C. Tintarev, \emph{On a version of Trudinger--Moser inequality with M\"obius shift invariance\text}, Calc. Var. Partial Differential Equations, {\bf 39} (2010), no. 1-2, 203--212.

\bibitem{AY}
Adimurthi, and Y. Yang, \emph{An interpolation of Hardy inequality and Trundinger--Moser inequality in $\R^N$ and its applications\text}, Int. Math. Res. Not. IMRN 2010, no. 13, 2394--2426. 

\bibitem{Ba}
A. Baernstein, \emph{A unified approach to symmetrization\text}, in: Partial Differential Equations of Elliptic Type, Cortona, 1992, in: Sympos. Math., vol. XXXV, 1994, pp. 47--91.

\bibitem{BFL}
R. D. Benguria, R. L. Frank, and M. Loss, \emph{The sharp constant in the Hardy--Sobolev--Maz'ya inequality in the three dimensional upper half--space\text}, Math. Res. Lett., {\bf 15} (2008), no. 4, 613--622.

\bibitem{Brock}
F. Brock, \emph{A general rearrangement inequality \`a la Hardy--Littlewood\text}, J. Inequal. Appl., {\bf 5} (2000) 309--320.

\bibitem{CC}
L. Carleson, and S. Y. A. Chang, \emph{On the existence of an extremal function for an inequality of J. Moser\text}, Bull. Sci. Math., {\bf 110} (1986) 113-127.

\bibitem{CR15}
G. Csat\'o, and P. Roy, \emph{Extremal functions for the singular Moser-Trudinger inequality in $2$ dimensions\text}, Calc. Var. Partial Differential Equations, {\bf 54} (2015), no. 2, 2341--2366. 

\bibitem{CNR}
G. Csat\'o, V. H. Nguyen, and P. Roy, \emph{Extremals for the singular Moser--Trudinger inequality via $n$-harmonic transplantation\text}, arXiv:1801.03932v3.

\bibitem{CR16}
G. Csat\'o, and P. Roy, \emph{Singular Moser--Trudinger inequality on simply connected domains\text}, Comm. Partial Differential Equations, {\bf 41} (2016), no. 5, 838--847.

\bibitem{DeFOR}
D. G. De Figueiredo, J. M. do \'O, and B. Ruf, \emph{On an inequality by N. Trudinger and J. Moser and related elliptic equations\text}, Comm. Pure Appl. Math., {\bf 55} (2002), no. 2, 135--152.

\bibitem{F}
M. Flucher, \emph{Extremal functions for the Trudinger-Moser inequality in $2$ dimensions\text} 
Comment. Math. Helv., {\bf 67} (1992) 471--497.

\bibitem{GS}
J. Garcia--Meli\'an, and J. Sabina de Lis, \emph{Maximum and comparison principles for operators involving the $p$-Laplacian\text}, J. Math. Anal. Appl., {\bf 218} (1998), no. 1, 49--65.

\bibitem{Hou}
S. Hou, \emph{Extremal functions for a singular Hardy--Moser--Trudinger inequality\text}, preprint, arXiv:1908.03982v1.

\bibitem{LL}
N. Lam, and G. Lu, \emph{A new approach to sharp Moser--Trudinger and Adams type inequalities: A rearrangement--free argument\text}, J. Differential Equations, {\bf 255} (2013) 298--325.

\bibitem{Li2001}
Y. Li, \emph{Moser--Trudinger inequaity on compact Riemannian manifolds of dimension two\text}, J. Partial Differ. Equa., {\bf 14} (2001) 163-192.

\bibitem{Li2005}
Y. Li, \emph{Extremal functions for the Moser-Trudinger inequalities on compact Riemannian manifolds\text}, Sci. China Ser. A, {\bf 48} (2005) 618--648.

\bibitem{LR}
Y. Li, and B. Ruf, \emph{A sharp Trudinger-Moser type inequality for unbounded domains in $\R^n$\text}, Indiana Univ. Math. J., {\bf 57} (2008) 451--480.

\bibitem{LLY}
J. Li, G. Lu, and Q. Yang, \emph{Fourier analysis and optimal Hardy--Adams inequalities on hyperbolic spaces of any even dimension\text}, Adv. Math., {\bf 333} (2018) 350--385.

\bibitem{Lin1996}
K. Lin, \emph{Extremal functions for Moser's inequality\text}, Trans. Amer. Math. Soc., {\bf 348} (1996) 2663--2671. 

\bibitem{LY16}
G. Lu, and Q. Yang, \emph{A sharp Trudinger--Moser inequality on any bounded and convex planar domain\text}, Calc. Var. Partial Differential Equations, {\bf 55} (2016), no. 6, Art. 153, 16 pp.

\bibitem{LY}
G. Lu, and Q. Yang, \emph{Sharp Hardy--Adams inequalities for bi--Laplacian on hyperbolic space of dimension four\text}, Adv. Math., {\bf 319} (2017) 567--598.

\bibitem{MS}
G. Mancini, and K. Sandeep, \emph{Moser-Trudinger inequality on conformal discs\text}, Commun. Contemp. Math., {\bf 12} (2010), no. 6, 1055--1068.

\bibitem{MST}
G. Mancini, K. Sandeep, and C. Tintarev, \emph{Trudinger--Moser inequality in the hyperbolic space $\mathbb H^n$\text}, Adv. Nonlinear Anal., {\bf 2} (2013), no. 3, 309--324.

\bibitem{Marcus}
M. Marcus, V.J. Mizel, and Y. Pinchover, \emph{On the best constant for Hardy’s inequality in $\R^n$\text}, Trans. Amer. Math. Soc., {\bf 350} (1998), no. 8, 3237--3255.

\bibitem{Martinazzi}
L. Martinazzi, \emph{Fractional Adams--Moser--Trudinger type inequalities\text}, Nonlinear Anal., {\bf 127} (2015) 263--278.

\bibitem{Mat}
T. Matskewich, and P. Sobolevskii, \emph{The best possible constant in a generalized Hardy’s inequality for convex domains in $\R^n$\text}, Nonlinear Anal., {\bf 28} (1997), no. 9, 1601--1610.

\bibitem{Maz'ya}
V. G. Maz'ya, \emph{Sobolev spaces\text}, Springer Verlag, Berlin, New York, 1985.

\bibitem{M}
J. Moser, \emph{A sharp form of an inequality by N. Trudinger\text}, Indiana Univ. Math. J., {\bf 20} (1970/71) 1077-1092.

\bibitem{NgoNguyen2016}
Q. A. Ngo, and V. H. Nguyen, \emph{Sharp Adams--Moser--Trudinger type inequalities in the hyperbolic spaces\text}, to appear Revista Matem\'atica Iberoamericana, arXiv:1606.07094.

\bibitem{NguyenHS}
V. H. Nguyen, \emph{The sharp Poincar\'e--Sobolev type inequalities in the hyperbolic spaces $\mathbb H^n$\text}, J. Math. Anal. Appl., {\bf 462} (2018), no. 2, 1570--1584.

\bibitem{NguyenT}
V. H. Nguyen, \emph{Improved Moser--Trudinger inequality of Tintarev type in dimension $n$ and the existence of its extremal functions\text}, Ann. Global Anal. Geom., {\bf 54} (2018), no. 2, 237--256.

\bibitem{NguyenMT}
V. H. Nguyen, \emph{Improved Moser--Trudinger type inequalities in the hyperbolic space $\mathbb H^n$\text}, Nonlinear Anal., {\bf 168} (2018) 67--80.

\bibitem{Nguyen}
V. H. Nguyen, \emph{Improved singular Moser--Trudinger inequalities and their extremal functions\text}, to appear in Potential Analysis.

\bibitem{Nguyennew}
V. H. Nguyen, \emph{The Hardy--Moser--Trudinger inequality via the transplantation of Green functions\text}, minor revision in Communications on Pure and Applied Analysis.

\bibitem{PT}
Y. Pinchover, and K. Tintarev, \emph{Ground state alternative for $p-$Laplacian with potential term\text}, Calc. Var. and Partial Differential Equations, {\bf 28} (2007) 179--201.

\bibitem{P}
S. I. Poho${\rm \check{z}}$aev, \emph{On the eigenfunctions of the equation $\Delta u + \lambda f(u) = 0$\text}, (Russian), Dokl. Akad. Nauk. SSSR, {\bf 165} (1965) 36-39.

\bibitem{R}
B. Ruf, \emph{A sharp Trudinger-Moser type inequality for unbounded domains in $\R^2$\text}, J. Funct. Anal., {\bf 219} (2005) 340--367.

\bibitem{TT}
A. Tertikas and C. Tintarev, \emph{On existence of minimizers for the Hardy--Sobolev--Maz'ya inequality\text}, Ann. Mat. Pura Appl. (4), {\bf 186} (2007), no. 4, 645--662.

\bibitem{Serrin}
J. Serrin, \emph{Local behavior of solutions of quasilinear equations\text}, Acta Math., {\bf 111} (1964) 247--302.

\bibitem{Tin}
C. Tintarev, \emph{Trudinger--Moser inequality with remainder terms\text}, J. Funct. Anal., {\bf 266} (2014) 55--66.

\bibitem{Tol}
P. Tolksdorf, \emph{Regularity for a more general class of quasilinear elliptic equations\text}, J. Differential Equations, {\bf 51} (1984), no. 1, 126--150.

\bibitem{T}
N. S. Trudinger, \emph{On imbedding into Orlicz spaces and some applications\text}, J. Math. Mech., {\bf 17} (1967) 473-483.

\bibitem{WangYe}
G. Wang, and D. Ye, \emph{A Hardy--Moser--Trudinger inequality\text} Adv. Math., {\bf 230} (2012), no. 1, 294--320.

\bibitem{W19a}
X. Wang, \emph{Improved Hardy--Adams inequality on hyperbolic space of dimension four\text}, Nonlinear Anal., {\bf 182} (2019) 45--56.

\bibitem{W19b}
X. Wang, \emph{Singular Hardy--Moser--Trudinger inequality and the existence of extremals on the unit disc\text}, Commun. Pure Appl. Anal., {\bf 18} (2019) 2717--2733.


\bibitem{Yang06}
Y. Yang, \emph{A sharp form of Moser--Trudinger inequality in high dimension\text}, J. Funct. Anal., {\bf 239} (2006), no. 1, 100--126.

\bibitem{Yang}
Y. Yang, \emph{A sharp form of the Moser--Trudinger inequality on a compact Riemannian surface\text}, Trans. Amer. Math. Soc., {\bf 359} (2007), no. 12, 5761--5776.

\bibitem{YSK}
Q. Yang, D. Su and Y. Kong, \emph{Sharp Moser--Trudinger inequalities on Riemannian manifolds with negative curvature\text}, Ann. Mat. Pura Appl., {\bf 195} (2016) 459--471.

\bibitem{YZ}
Y. Yang, and X. Zhu, \emph{An improved Hardy--Trudinger--Moser inequality\text}, Ann. Global Anal. Geom., {\bf 49} (2016), no. 1, 23--41.

\bibitem{YZsingular}
Y. Yang, and X. Zhu, \emph{Blow-up analysis concerning singular Trudinger--Moser inequalities in dimension two\text}, J. Funct. Anal., {\bf 272} (2017), no. 8, 3347--3374.

\bibitem{Y}
V. I. Yudovi${\rm \check{c}}$, \emph{Some estimates connected with integral operators and with solutions of elliptic equations\text}, (Russian), Dokl. Akad. Nauk. SSSR, {\bf 138} (1961) 805-808.


\end{thebibliography}
\end{document}